\documentclass[letterpaper,10pt,reqno,onefignum,onetabnum]{amsart}
\usepackage[english]{babel}
\usepackage{amsmath}
\usepackage{amsthm}
\usepackage{verbatim}
\usepackage{mathrsfs}
\usepackage{bm} 
\usepackage[dvipsnames]{xcolor}
\usepackage{hyperref}
\hypersetup{
     colorlinks = true,
     linkcolor = OliveGreen,
     anchorcolor = OliveGreen,
     citecolor = OliveGreen,
     filecolor = OliveGreen,
     urlcolor = OliveGreen
     }
\usepackage{algorithm}
\usepackage{algorithmic}
\usepackage{float}
\usepackage{lipsum}
\usepackage{amsfonts}
\usepackage{amssymb}
\usepackage{graphicx}
\usepackage{epstopdf}
\usepackage{cases}
\usepackage{multirow}
\ifpdf
  \DeclareGraphicsExtensions{.eps,.pdf,.png,.jpg}
\else
  \DeclareGraphicsExtensions{.eps}
\fi
\usepackage{verbatim}
\usepackage{mathrsfs}
\usepackage{bm}

\newtheorem{teo}{Theorem}[section]
\newtheorem{prop}[teo]{Proposition}
\newtheorem{lem}[teo]{Lemma}

\newtheorem{pro}[teo]{Problem}

\newtheorem{rem}[teo]{Remark}
\newtheorem{ejem}[teo]{Example}

\usepackage{lipsum}
\usepackage{amsfonts}
\usepackage{amssymb}
\usepackage{graphicx}
\usepackage{epstopdf}
            
\usepackage{cases}
\usepackage{multirow}
\usepackage{algorithmic}
\usepackage{tikz}
\usetikzlibrary{matrix}
\usetikzlibrary{arrows}
\ifpdf
  \DeclareGraphicsExtensions{.eps,.pdf,.png,.jpg}
\else
  \DeclareGraphicsExtensions{.eps}
\fi
\usepackage{verbatim}
\usepackage{mathrsfs}
\usepackage{bm} 
\usepackage{color}
\usepackage{caption}
\usepackage{subcaption}
\usepackage{enumitem}

\newcommand{\N}{\mathbb N}

\newcommand{\R}{\mathbb R}

\renewcommand{\H}{\mathcal{H}}
\newcommand{\G}{\mathcal G}

\newcommand{\hh}{\bm{\mathcal{H}}}
\newcommand{\HH}{{\bm{\mathcal{H}}}}
\newcommand{\vv}{{\bm V}}
\newcommand{\VV}{{\bm V}}
\newcommand{\Ss}{{\bm S}}
\newcommand{\TT}{{\bm T}}
\newcommand{\MM}{{\bm M}}
\newcommand{\WW}{{\bm W}}
\newcommand{\Id}{{\bf Id}}
\newcommand{\id}{\textnormal{Id}}
\newcommand{\x}{\bm x}

\newcommand{\s}{\sigma}
\newcommand{\T}{\tau}

\newcommand{\ran}{\textnormal{ran}\,}
\newcommand{\dom}{\textnormal{dom}\,}

\newcommand{\zer}{\textnormal{zer}}
\newcommand{\fix}{\textnormal{Fix}\,}
\newcommand{\gra}{\textnormal{gra}\,}

\newcommand{\rvv}{{\ran\, \vv}}
\newcommand{\rv}{{\ran\, V}}
\newcommand{\argm}[1]{\underset{#1}{\argmin\, }}

\newcommand{\scal}[2]{{\left\langle{{#1}\mid{#2}}\right\rangle}}
\newcommand{\pscal}[2]{\langle\!\langle{#1}\mid{#2}\rangle\!\rangle}

\newcommand{\menge}[2]{\big\{{#1}~\big |~{#2}\big\}} 
\newcommand{\pinf}{\ensuremath{{+\infty}}}

\newcommand{\RR}{\ensuremath{\mathbb{R}}}

\newcommand{\RPP}{\ensuremath{\left]0,+\infty\right[}}

\newcommand{\RX}{\ensuremath{\left]-\infty,+\infty\right]}}

\newcommand{\prox}{\ensuremath{\text{\rm prox}}}
\newcommand{\weakly}{\ensuremath{\:\rightharpoonup\:}}
\usepackage{geometry}
\numberwithin{equation}{section}

\DeclareFontEncoding{FMS}{}{}
\DeclareFontSubstitution{FMS}{futm}{m}{n}
\DeclareFontEncoding{FMX}{}{}
\DeclareFontSubstitution{FMX}{futm}{m}{n}
\DeclareSymbolFont{fouriersymbols}{FMS}{futm}{m}{n}
\DeclareSymbolFont{fourierlargesymbols}{FMX}{futm}{m}{n}
\DeclareMathDelimiter{\nr}{\mathord}{fouriersymbols}{152}{fourierlargesymbols}{147}

\DeclareMathOperator*{\argmin}{arg\,min}
\DeclareMathDelimiter{\nr}{\mathord}{fouriersymbols}{152}{fourierlargesymbols}{147}

\title[Primal-dual splittings in the range of 
linear operators]{Primal-dual splittings as fixed point iterations in the range of 
linear operators}

\author{Luis M. Brice\~no-Arias \& Fernando Rold\'an}
\address{Departamento de Matem\'{a}tica, Universidad T\'{e}cnica Federico Santa Mar\'{i}a, Avenida Espa\~{n}a 1680, Valpara\'{i}so, Chile}
\email{luis.briceno@usm.cl, fernando.roldan@usm.cl}


\begin{document}

\begin{abstract}
In this paper we study the relaxed primal-dual 
algorithm for solving composite 
monotone inclusions in real Hilbert spaces with
critical preconditioners. 
Our approach is based in new results on the asymptotic behaviour of
Krasnosel'ski\u{\i}-Mann (KM) iterations defined in the range of 
monotone self-adjoint linear operators. These results generalize
the convergence of classical KM iterations aiming at approximating 
fixed points. 
We prove that the relaxed primal-dual algorithm 
with critical preconditioners define KM iterations in the range of a 
particular monotone self-adjoint linear operator with non-trivial kernel.
We then deduce from our fixed point approach that the shadows of 
primal-dual iterates on the range 
of the linear operator converges weakly to some point in this vector 
subspace from which we obtain a 
solution. This generalizes \cite[Theorem~3.3]{condat} to infinite 
dimensional relaxed primal-dual monotone inclusions involving critical 
preconditioners.
The Douglas-Rachford splitting (DRS) is interpreted as a particular 
instance of the primal-dual algorithm 
when the step-sizes are critical and we recover classical results from 
this new perspective. We 
implement the relaxed primal-dual algorithm with critical 
preconditioners in total variation reconstruction and we 
illustrate its flexibility and efficiency.
\par
\bigskip

\end{abstract}

\maketitle
\section{Introduction}
In this paper we provide a theoretical study of the relaxed primal-dual 
splitting
\cite{vu} for solving the following composite monotone inclusion. 
\begin{pro}\label{prob2}
	Let $\H$ and $\G$ be a real Hilbert spaces, let $A : \H \rightarrow 2^{\H}$ and $B : \G 
	\rightarrow 2^{\G}$  be  maximally monotone operators, and let $L : \H \rightarrow \G$ be a 
	linear 
	bounded operator. The problem is to find $(\hat{x},\hat{u}) \in \bm{Z}$,
	where
	\begin{equation}\label{Z}
	\bm{Z}=\menge{(x,u) \in \H \times \G}{0 \in Ax+ L^*u,\:
		0 \in B^{-1}u-Lx}
	\end{equation}
	is assumed to be non-empty.
\end{pro}
It follows from \cite[Proposition~2.8]{skew} that 
any solution $(\hat{x},\hat{u})$ to Problem~\ref{prob2} satisfies 
that $\hat{x}$ is a solution to the primal inclusion
\begin{equation}
\label{e:priminc}
\text{find}\quad x\in\H\quad\text{such that}\quad 0\in Ax+L^*BLx
\end{equation}
and $\hat{u}$ is solution to the dual inclusion
\begin{equation}
\label{e:dualinc}
\text{find}\quad u\in\G\quad\text{such that}\quad 0\in 
B^{-1}u-LA^{-1}(-L^*u).
\end{equation}
Conversely, if $\hat{x}$ is a solution to \eqref{e:priminc} then
there exists $\tilde{u}$ solution to \eqref{e:dualinc} such that
$(\hat{x},\tilde{u})\in\bm{Z}$ and the dual argument also holds.
In the particular case when $A=\partial f$ and 
$B=\partial g$ for lower 
semicontinuous  
convex proper functions $f\colon\H\to\RX$ and  $g\colon\G\to\RX$, 
we have that $\bm{Z}\subset \mathcal{P}\times\mathcal{D}$,
where $\mathcal{P}$ is the set of solutions to 
the convex optimization 
problem 
\begin{equation}
\label{e:optim}
\min_{x\in\H}f(x)+g(Lx)
\end{equation}
and $\mathcal{D}$ is the set of solutions to its Fenchel-Rockafellar 
dual
\begin{equation}
\label{e:dual}
\min_{u\in\G}g^*(u)+f^*(-L^*u).
\end{equation}
Problem~\ref{prob2} and its particular 
optimization case model a wide class of problems in engineering 
going from from mechanical problems 
\cite{gabay83,GlowinskyMorrocco75,Goldstein64}, differential 
inclusions \cite{Aubin,Showalter}, game theory \cite{Nash}, 
image processing 
problems as image restoration and denoising 
\cite{ChambLions97,Pustelnik2019,Daube04}, traffic theory 
\cite{Nets1,Fuku96,GafniBert84}, 
among 
other disciplines. 

 In the last years, several algorithms have been proposed in 
order to solve Problem~\ref{prob2} and some generalizations 
involving cocoercive operators \cite{skew,vu,yuan}.
One of the most used, is the algorithm proposed in \cite{vuvarmet}
(see also \cite{vu,BOT2,bot2013} and \cite{CPockDiag,cp} in the 
context of \eqref{e:optim}), which iterates
\begin{equation}
\label{primaldualalgorithm}
(\forall n\in\N)\quad 
	\begin{array}{l}
	\left\lfloor
	\begin{array}{l}
p_{n+1}=J_{\Upsilon A}(x_n-\Upsilon L^*u_n)\\
q_{n+1}=J_{\Sigma B^{-1}}\big( u_n+\Sigma L(2p_{n+1}-x_n)\big)\\
(x_{n+1},y_{n+1})=(1-\lambda_n)(x_n,y_n)+\lambda_n(p_{n+1},q_{n+1}),
\end{array}
\right.
\end{array}
\end{equation}
where $(x_0 , u_0) \in \H\times\G$, $(\lambda_n)_{n\in\N}\subset 
\left]0,2\right[$,
and preconditioners $\Upsilon\colon\H\to\H$ and 
$\Sigma\colon\G\to\G$ are strongly 
monotone self-adjoint linear bounded operators such that
$\|\sqrt{\Sigma} 
L\sqrt{\Upsilon}\|<1$. 
It turns out that \eqref{primaldualalgorithm}
corresponds to the relaxed proximal-point algorithm 
\cite{rockafellar1976,martinet1970} associated 
to the operator $\vv^{-1}\bm{M}$,
where 
$\bm{M}\colon(x,u)\mapsto (Ax+L^*u)\times (B^{-1}u-Lx)$ is
maximally monotone in $\H\oplus\G$
 \cite[Proposition~2.7(iii)]{skew}
and the self-adjoint linear
 bounded operator
\begin{equation} \label{vv}
\vv: \hh \rightarrow \hh :(x,u) \mapsto \left( 
\Upsilon^{-1}x-L^*u,\Sigma^{-1}u-Lx\right)
\end{equation} 
is strongly monotone if $\|\sqrt{\Sigma} 
L\sqrt{\Upsilon}\|<1$. Hence, 
$\vv^{-1}\bm{M}$ is maximally monotone in the real Hilbert space
$(\H\times\G,\scal{\cdot}{\vv\cdot})$ and the convergence is a 
consequence of \cite{rockafellar1976,martinet1970}.
Note that $J_{\bm{M}}$ is also firmly 
nonexpansive in $\H\oplus\G$, however it has no 
explicit computation.
Non-standard metrics are widely used not only 
to obtain explicit resolvent computations but also to accelerate 
algorithms 
\cite{Siopt2,condat,yuan,vu,Davisconvrate,combvu}. In the 
presence of critical preconditioners, i.e., $\|\sqrt{\Sigma} 
L\sqrt{\Upsilon}\|=1$, the non-standard metric approach 
fails since $\ker\vv\neq\{0\}$ and, hence,  
$\scal{\cdot}{\vv\cdot}$ is not an inner product in $\H\times\G$.
Hence, the convergence of \eqref{primaldualalgorithm} when 
$\|\sqrt{\Sigma} 
L\sqrt{\Upsilon}\|\le 1$ in the 
context of Problem~\ref{prob2} is an open question, which is partially 
answered in \cite[Theorem~3.3]{condat} for solving \eqref{e:optim}
in a finite dimensional setting when $\Sigma=\s\id$ and 
$\Upsilon=\T\id$. In the particular case when $\lambda_n\equiv 1$,
the weak convergence of \eqref{primaldualalgorithm} when 
$\|\sqrt{\Sigma} L\sqrt{\Upsilon}\|\le 1$ is deduced in \cite{SDR} from 
an alternative formulation of \eqref{primaldualalgorithm}. 
This formulation in the case when $L=\id$, generates primal-dual 
iterates in the graph of $A$ and, therefore, the argument does not 
hold when relaxation steps are included in general. 

In this paper we generalize \cite[Theorem~3.3]{condat} to the 
monotone inclusion in Problem~\ref{prob2} in the infinite dimensional 
setting with critical preconditioners. Our approach is based on
a fixed point theory restricted to $(\rvv,\scal{\cdot}{\vv\cdot})$, which is 
a real Hilbert space under the condition $\rvv$ closed. 
We obtain the weak convergence of  Krasnosel'ski\u{\i}-Mann 
iterations
governed by firmly quasinonexpansive and averaged operators
in $(\rvv,\scal{\cdot}{\vv\cdot})$, which generalizes 
\cite[Theorem~5.2(i)]{ipa} and \cite[Proposition 5.16]{1}.
This result is interesting in its own right.
It is worth to notice that
most of known algorithms can be seen as fixed point iterations of
operators belonging to the previous classes 
\cite{BOT2,bot2013,skew,condat,yuan}. 
 Our approach gives new insights on primal-dual algorithms:
the convergence of primal-dual iterates in $\hh$ follows from the 
convergence of
their \textit{shadows} in $\ran\,\vv$. 

We also provide a detailed analysis of the case $L=\id$ and 
relations of  primal-dual algorithms with the relaxed 
Douglas-Rachford splitting (DRS) 
\cite{eckstein-bertsekas,lions-mercier79}. 
We give a primal-dual 
version of DRS derived from
\eqref{primaldualalgorithm} when $L=\id$
 and we recover the weak convergence of an auxiliary sequence 
whose \textit{primal-dual shadow} is a solution to Problem~\ref{prob2},
as in \cite{eckstein-bertsekas,lions-mercier79}.

We finish this paper by providing a numerical experiment on 
total variation image 
reconstruction, in which the advantages of using critical 
preconditioners and relaxation steps are illustrated.

The paper is organized as follows. In Section~\ref{sec:notation} we set our notation 
and some preliminaries. In Section~\ref{sec:fixedpoints} we study the 
fixed point problem on the range of linear operators
and we provide conditions for the convergence of fixed point iterations 
governed by firmly quasinonexpansive 
or  averaged nonexpansive operators. In Section~\ref{sec:primaldual} 
we apply fixed point results to the particular case of primal-dual 
monotone 
inclusions and we provide several 
connections with other results in the literature. In Section~\ref{sec:DR}, we study
in detail the particular case  when $L=\id$, which
is connected with Douglas--Rachford splitting. Finally,
in Section~\ref{sec:TV} we provide numerical experiments in 
image processing.

\section{Notation and Preliminaries}
In this section we first provide our notation and some preliminaries. 
Next we study a specific fixed point problem involving the range of a 
self-adjoint monotone linear operator and we provide the convergence 
of a Krasnosel'ski\u{\i}--Mann iteration in this vector subspace, which 
is 
interesting in its own right.  
The weak convergence of \eqref{primaldualalgorithm}
with critical preconditioners is derived from previous fixed point 
analysis.

\subsection{Notation and preliminaries}
\label{sec:notation}
Throughout this paper $\H$ and $\G$ are real Hilbert spaces. We denote the scalar 
product by $\scal{\cdot}{\cdot}$ and the associated norm by 
$\|\cdot \|.$ 
The symbols $\weakly$ and $\to$ denotes the weak and strong 
convergence, respectively.
Given a linear bounded 
operator $L:\H \to \G$, we denote its adjoint by $L^*\colon\G\to\H$, its 
kernel by $\ker L$,
and its range by $\ran L$. $\id$ denotes the identity operator on $\H$. 
Let $D\subset \H$ be non-empty and let $T: D \rightarrow \H$. 
The set of fixed points of $T$ is given by 
$\fix T = \menge{x \in D}{x=Tx}$.
Let $\beta \in \left]0,+\infty\right[$. The operator $T$ is $\beta-$cocoercive if 
\begin{equation} \label{def:coco}
(\forall x \in D) (\forall y \in D)\quad \langle x-y \mid Tx-Ty \rangle \geq \beta \|Tx - Ty 
\|^2,
\end{equation}
it is $\beta-$strongly monotone if 
\begin{equation}\label{def:strongmono}
(\forall x \in D) (\forall y \in D)\ \ \langle x-y \mid Tx-Ty \rangle \geq \beta \|x - y 
\|^2,\end{equation}
it is nonexpansive if 
\begin{equation}\label{def:nonexpansive}
(\forall x \in D) (\forall y \in D)\ \ \|Tx - Ty \|\leq \|x-y\|,
\end{equation}
it is quasinonexpansive if 
\begin{equation}\label{def:quasiexpansive}
(\forall x \in D) (\forall y \in \fix T)\ \ \|Tx - y \| \leq 
\|x-y\|,
\end{equation}
and it is firmly quasinonexpansive (or class $\mathfrak{T}$) if 
\begin{equation}\label{def:quasifirmnonexpansive}
(\forall x \in D) (\forall y \in \fix T)\ \ \|Tx - y \|^2 \leq 
\|x-y\|^2-\|Tx-x\|^2.
\end{equation}
Let $\alpha \in\, ]0,1[$. The operator $T$ is $\alpha-$averaged 
nonexpansive if 
$T=(1-\alpha)\id + \alpha R$ for some nonexpansive operator 
$R: \H \rightarrow \H$,
and $T$ is firmly nonexpansive if it is $\frac{1}{2}-$averaged 
nonexpansive.

Given a 
self-adjoint monotone linear bounded
operator $V\colon\H\to\H$, 
we denote $\scal{\cdot}{\cdot}_V=\scal{\cdot}{V\cdot}$, which is 
bilinear,  positive 
semi-definite, and 
symmetric. Moreover, there exists a self-adjoint monotone linear 
bounded operator $\sqrt{V}\colon\H\to\H$
such that 
\begin{equation}
\label{e:sqrtV}
V=\sqrt{V}\sqrt{V},\quad (\forall x\in\H)\quad 
\scal{x}{Vx}=\|\sqrt{V}x\|^2,
\end{equation}
and $\ran V=\ran\sqrt{V}$.
In addition, if $V$ is strongly monotone, 
$\scal{\cdot}{\cdot}_V$ defines an inner product on $\H$ and we 
denote by $\|\cdot\|_V=\sqrt{\scal{\cdot}{\cdot}_V}$ the induced norm.

Let $A:\H \rightarrow 2^{\H}$ be a set-valued operator. The domain, 
range, and graph of $A$ are 
$\dom\, A = \menge{x \in \H}{Ax \neq  \varnothing}$, 
$\ran\, A = \menge{u \in \H}{(\exists x \in \H)\, u \in Ax}$, 
and $\gra 
A = \menge{(x,u) \in \H \times \H}{u \in Ax}$,
respectively. 
The set of zeros of $A$ is  $\zer A = 
\menge{x \in \H}{0 \in Ax}$, the inverse of $A$ is $A^{-1}\colon \H  \to 2^\H \colon u 
\mapsto 
\menge{x \in \H}{u \in Ax}$, and the resolvent of $A$ is 
$J_A=(\id+A)^{-1}$. We have $\zer A=\fix J_A$.
 The operator $A$  is monotone if 
\begin{equation}\label{def:monotone}
(\forall (x,u) \in \gra A) (\forall (y,v) \in \gra A)\quad \scal{x-y}{u-v} \geq 0
\end{equation}
and it is maximally monotone if it is monotone and there exists no 
monotone operator $B :\H\to  2^{\H}$ such that $\gra B$ properly 
contains $\gra A$, i.e., for every $(x,u) \in \H \times \H$,
\begin{equation} \label{def:maxmonotone}
(x,u) \in \gra A \quad \Leftrightarrow\quad  (\forall (y,v) \in \gra A)\ \ 
\langle x-y \mid u-v \rangle \geq 0.
\end{equation}
 Let $C$ be a non-empty subset of $\H$ and let $\big( x_n \big)_{n \in 
 \N}$ be a sequence 
 in $\H.$ Then  $\big( x_n \big)_{n \in \N}$ is Fej\'er monotone with 
 respect to $C$ if 
\begin{equation}
(\forall x \in C)(\forall n \in \mathbb{N}) \quad\left\|x_{n+1}-x\right\| \leqslant\left\|x_{n}-x\right\|.
\end{equation}
Let $D$ be a non-empty weakly sequentially closed subset of
$\mathcal{H},$ let $T : D \rightarrow \mathcal{H},$ and let $u \in \mathcal{H} .$ 
Then $T$ is demiclosed at $u$ in $(\H , \scal{\cdot}{\cdot})$ if, for every
sequence $\left(x_{n}\right)_{n \in \mathbb{N}}$ in $D$ and every $x \in D$ 
such that $x_{n} \weakly x$ and $T x_{n} \to u$ in $(\H , \scal{\cdot}{\cdot}),$ we 
have have $T x=u .$ In addition, $T$ is demiclosed if it is demiclosed 
at every point in $D.$ 

We denote by $\Gamma_0(\H)$ the class of proper lower 
semicontinuous convex functions $f\colon\H\to\RX$. Let 
$f\in\Gamma_0(\H)$.
The Fenchel conjugate of $f$ is 
defined by $f^*\colon u\mapsto \sup_{x\in\H}(\scal{x}{u}-f(x))$, which 
is a function in $\Gamma_0(\H)$,
the subdifferential of $f$ is the maximally monotone operator
$$\partial f\colon x\mapsto \menge{u\in\H}{(\forall y\in\H)\:\: 
f(x)+\scal{y-x}{u}\le f(y)},$$
$(\partial f)^{-1}=\partial f^*$,
and we have  that $\zer\partial f$ is the set of 
minimizers of $f$, which is denoted by $\arg\min_{x\in \H}f$. 
Given a strongly monotone self-adjoint linear operator 
$\varUpsilon\colon\H\to\H$, we denote by 
\begin{equation}
	\label{e:prox}
\prox^{\varUpsilon}_{f}\colon 
x\mapsto\argm{y\in\H}\big(f(y)+\frac{1}{2}\|x-y\|_{\varUpsilon}^2\big),
\end{equation}
and by $\prox_{f}=\prox^{\id}_{f}$. We have 
$\prox^{\varUpsilon}_f=J_{\varUpsilon^{-1}\partial f}$ 
\cite[Proposition~24.24(i)]{1}
and it is single valued since 
the objective function in 
\eqref{e:prox} is strongly convex. Moreover, it follows from 
\cite[Proposition~24.24]{1} that
\begin{equation}
\label{e:Moreau_nonsme}
\prox^{\varUpsilon}_{f}=
\id-\varUpsilon^{-1}\,  
\prox^{\varUpsilon^{-1}}_{f^*}\, 
\varUpsilon=\varUpsilon^{-1}\,(\id-  
\prox^{\varUpsilon^{-1}}_{f^*})\, 
\varUpsilon.
\end{equation}
Given a non-empty closed convex set $C\subset\H$, we denote by 
$P_C$ the projection onto $C$ and by
$\iota_C\in\Gamma_0(\H)$ the indicator function of $C$, which 
takes the value $0$ in $C$ and $\pinf$ otherwise. 
For further properties of monotone operators,
nonexpansive mappings, and convex analysis, the 
reader is referred to \cite{1}.

The following result allows us to define algorithms in a real Hilbert 
space defined 
by the range of non-invertible self-adjoint linear bounded operators.
The result is a direct consequence of 
\cite[Fact~2.26]{1} and \eqref{e:sqrtV}.
\begin{prop}
\label{prophilbert}
Let $V: \H \rightarrow \H$ be a monotone self-adjoint linear bounded 
operator. The following statements are equivalent.
\begin{enumerate}
\item \label{prop:closedrange11}
$\rv$ is closed.
\item \label{prop:closedrange21}$(\exists \alpha > 0)(\forall x \in \ran\, 
V )$
\begin{equation}
\label{propiedadV1}
\quad\langle Vx\mid x \rangle \geq \alpha \|x\|^2.
\end{equation}
\end{enumerate}
 Moreover, if \ref{prop:closedrange11} or \ref{prop:closedrange21} 
 holds, then $ (\ran\, V,\scal{\cdot}{\cdot}_{V})$ is a real Hilbert space.
\end{prop}

The following example exhibits a monotone self-adjoint linear 
bounded operator whose range is not closed,
illustrating that assumption $\ran\, V$ closed is 
not redundant in our setting.
\begin{ejem}
Let $\ell^2(\R)$ be the real Hilbert space defined by square summable 
sequences in $\RR$ endowed by the inner product
$\scal{\cdot}{\cdot}\colon
(x,y)\mapsto \sum_{j\ge1} x_j\, y_j$
and consider the 
monotone self-adjoint bounded linear
operator 
\begin{displaymath}
V: \ell^2(\R) \to \ell^2(\R)\colon (x_n)_{n\in\N\setminus\{0\}} \mapsto 
\left( x_1, \frac{x_2}{2}, 
\frac{x_3}{3}, \ldots\right).
\end{displaymath}
By considering the sequence $(x^n)_{n \in \N} \subset \ell^2(\R)$ 
defined by $x_j^n=1$ for $j\le n$ and $x_j^n=0$ for $j> n$...
we have $Vx^n\to y=(1/j)_{j \in 
\N\setminus\{0\}}\in\ell^2(\R)$  as $n\to+\infty$, and
$y\not\in \rv$, which implies that $\rv$  is not 
closed.
\end{ejem}

	\subsection{Fixed points in the range of linear operators}
	\label{sec:fixedpoints}	
The following fixed point problem 
is the basis for the analysis of primal-dual algorithms. 
\begin{pro}\label{prob1}
Let $(\hh,\scal{\cdot}{\cdot})$ be a real Hilbert space, let $\vv : \hh 
\rightarrow \hh$  be a 
monotone
self-adjoint linear bounded operator such that $\rvv$ is 
closed, and let 
$\Ss\colon\hh\to\hh $ be such that 
$\fix \Ss\neq \varnothing$, that 
$\Ss=\Ss \circ P_{\rvv}$, and that 
$(P_{\rvv}\circ\Ss)|_{\rvv}$ is quasinonexpansive in $(\rvv, 
\scal{\cdot}{\cdot}_{\vv})$. 
The problem is to
\begin{equation} \label{problem1}
\textrm{find}\quad   \x \in \fix 
\Ss.
\end{equation}
\end{pro}
First observe that 
under the hypotheses on $\vv$,  
Proposition~\ref{prophilbert} asserts that $(\rvv, 
\scal{\cdot}{\cdot}_{\vv})$ is a real Hilbert space.
In the particular case when $\vv=\Id$, 
we have $\ran\, \vv = \hh$, $P_{\ran\vv}=\Id$, and
Problem~\ref{prob1} is solved in \cite{ipa} when 
$\Ss$ is 
firmly quasinonexpansive (or class $\mathfrak{T}$), and $\Id-\Ss$ is 
demiclosed at 
$\bm{0}$, and in \cite{1,opti04} when 
$\Ss$ is averaged nonexpansive.
In the case when $\vv$ is self-adjoint and strongly monotone, 
 we also have
$\ran\, \vv = \hh$, $P_{\ran\vv}=\Id$, and several approaches with 
non-standard metrics
are developed for solving Problem~\ref{prob1} by using contractive
assumptions on $\Ss$
(see, e.g., 
\cite{Siopt2,combvu,condat,Davisconvrate,yuan,CPockDiag,vu}).
In all cases, the problem is solved via the Krasnosel'ski\u{\i}--Mann 
iteration
\begin{equation}\label{fpi}
\bm{x}_0 \in \hh, \quad (\forall n \in \N) \quad 
\bm{x}_{n+1}=(1-\lambda_n)\bm{x}_{n}+\lambda_n \Ss \bm{x}_{n},
\end{equation}
where $(\lambda_n)_{n\in\N}$ is a strictly positive sequence. 
The main difference of Problem~\ref{prob1} with respect to previous 
literature is that the contractive property is only guaranteed for  
the \textit{shadow operator} $(P_{\rvv}\circ\Ss)|_{\rvv}$ on
$(\rvv, \scal{\cdot}{\cdot}_{\vv})$ without any further assumption on 
$\Ss$. 
In this section we obtain conditions for ensuring the convergence of 
the \textit{shadow sequence} $(P_{\rvv}\x_n)_{n\in\N}$  to a solution 
to Problem~\ref{prob1} from \eqref{fpi}. 
 We first need the following technical lemma.
\begin{lem} \label{equalsetsGEN}
	Let  $\bm{Q}\colon\hh\to\hh$ and let $\Ss\colon\hh\to\hh $ be 
	such that 
	$\fix \Ss\neq \varnothing$ and $\Ss=\Ss \circ \bm{Q}$.
	Then 
	$\Ss ( \fix (\bm{Q}\circ\Ss) )= \fix \Ss$ and, in particular, 
	$\fix (\bm{Q}\circ\Ss) \neq \varnothing$.
\end{lem}
\begin{proof}
	First, let $ \bm{x} \in \fix \Ss.$ Since $\Ss=\Ss \circ \bm{Q}$ 
	we have
	\begin{align*}
		\Ss \bm{x} = \bm{x}\quad  
		&\Leftrightarrow\quad   \Ss(\bm{Q} \bm{x}) = \bm{x}\\
		&\Rightarrow\quad \bm{Q}\circ \Ss (\bm{Q}\bm{x}) = 
		\bm{Q} \bm{x}
	\end{align*}
	hence $\bm{Q}\bm{x} \in \fix (\bm{Q}\circ \Ss)$ and 
	$\Ss(\bm{Q}\bm{x})=\bm{x}$. Thus, 
	$\bm{x} \in \Ss(\fix (\bm{Q}\circ \Ss))$ and we conclude 
	$\fix\Ss \subset \Ss ( \fix 
	(\bm{Q}\circ 
	\Ss))$. Conversely, let 
	$\bm{x} \in \fix (\bm{Q}\circ \Ss)$. Since $\Ss=\Ss \circ 
	\bm{Q}$, we have
	\begin{align*}
		\bm{Q} (\Ss\bm{x})=\bm{x}\quad 
		&\Rightarrow\quad    \Ss(\bm{Q} (\Ss\bm{x}))=\Ss\bm{x}\\
		&\Rightarrow\quad    \Ss(\Ss\bm{x})=\Ss\bm{x}\\
		&\Rightarrow\quad    \Ss\bm{x} \in \fix\Ss.
	\end{align*}
	Thus $ \Ss ( \fix (\bm{Q}\circ \Ss)) \subset \fix\Ss$ and the 
	result follows.\qed
\end{proof}
Now we prove that
Krasnosel'ski\u{\i}--Mann iterations defined by $\Ss$ approximate the 
solutions to 
Problem~\ref{prob1} via their \textit{shadows} in ${\ran\vv}$. 
We derive our result
for firmly quasinonexpansive (or class 
$\mathfrak{T}$) operators $\TT$ such that $\Id-\TT$ is demiclosed at 
$\bm{0}$, and for  
$\alpha-$averaged nonexpansive operators, for some 
$\alpha\in\left]0,1\right[$. 
\begin{prop} \label{teogenquasi}
In the context of Problem~\ref{prob1}, define
\begin{equation}\label{defT}
	\TT: \rvv \rightarrow \rvv : x \mapsto P_\rvv \circ \Ss x
\end{equation}
and consider the sequence $(\bm{x}_n)_{n \in 
	\mathbb{N}}$ 
	defined by the 
	recurrence
	\begin{equation}\label{FPSq}
	\bm{x}_0 \in \hh, \quad (\forall n \in \N)\quad 
	\bm{x}_{n+1}=(1-\lambda_n)\,\bm{x}_n+\lambda_n\Ss \bm{x}_n.
	\end{equation}
Moreover, suppose that one of the following holds:
\begin{enumerate}[label=(\roman*)]
\item \label{teogenquasiquasihyp1}
$\TT$ is firmly quasinonexpansive, $\Id-\TT$ is demiclosed at 
$\bm{0}$ in $(\rvv, \scal{\cdot}{\cdot}_{\vv})$, and	
 $(\lambda_n)_{n\in\N}$ is a sequence in 
	$\left[\varepsilon,2-\varepsilon\right]$ for some 
	$\varepsilon\in\left]0,1\right[$.
\item \label{teogenquasiquasihyp2} $\TT$ is $\alpha-$averaged
nonexpansive in $(\rvv, \scal{\cdot}{\cdot}_{\vv})$ for some 
$\alpha\in\left]0,1\right[$ and $(\lambda_n)_{n\in\N}$ 
is a sequence in 
$\left[0,1/\alpha\right]$ such that 
$\sum_{n\in\N}\lambda_n(1-\alpha\lambda_n)=\pinf$.
\end{enumerate}
	Then the following hold:
	\begin{enumerate}
		\item\label{teogenquasiquasi}$(P_{\rvv}\x_n)_{n\in\N}$ is F\'ejer 
		monotone 
		in $(\rvv, \langle \cdot \mid 
		\cdot 
		\rangle_{\vv})$ 
		with respect to $\fix\TT$.
		\item\label{teogenquasistrong} $(P_{\rvv}(\Ss\x_n-\x_n))_{n\in\N}$ converges strongly to 
		$\bm{0}$ in $(\rvv, \scal{\cdot}{\cdot}_{\vv})$.
		\item\label{teogenquasiconv} $(P_{\rvv}\x_n)_{n\in\N}$ converges weakly in $(\rvv, 
		\scal{\cdot}{\cdot}_{\vv})$ to 
		some $\hat{\bm{x}} \in \fix\TT$ and $\Ss\hat{\bm{x}}$ is a solution to 
		Problem~\ref{prob1}.
	\end{enumerate}
\end{prop}
\begin{proof}
Since $\vv$ is a monotone
bounded self-adjoint linear operator and $\rvv$ is closed, it follows from 
Proposition~\ref{prophilbert} that
$(\rvv, \langle \cdot \mid \cdot 
\rangle_{\vv})$ is a real Hilbert space.
Moreover, since $\Ss=\Ss\circ P_{\rvv}$ and 
$\fix \Ss\ne \varnothing$, Lemma~\ref{equalsetsGEN}
yields 
\begin{equation}
	\label{e:relfix}
\Ss ( \fix \TT )= \fix \Ss
\end{equation}
and, hence, $\fix \TT \neq 
\varnothing$. Therefore, since
$P_{\rvv}$ is linear, by defining, for every 
$n\in\N$,
 $\bm{y}_n=P_{\rvv}\x_{n}$, it follows from \eqref{FPSq} and 
$\TT=\TT\circ P_{\rvv}$
that
 \begin{equation}
 \bm{y}_0 \in \rvv, \quad (\forall n \in \N)\quad 
 \bm{y}_{n+1}=(1-\lambda_n)\,\bm{y}_n+\lambda_n\TT \bm{y}_n.
 \end{equation}
If we assume \ref{teogenquasiquasihyp1}, since 
$\inf_{n\in\N}\lambda_n(2-\lambda_n)\ge\varepsilon^2$, 
\ref{teogenquasiquasi} and \ref{teogenquasistrong} follow from 
\cite[Proposition~4.2]{ipa} in the error free  case. Finally, it follows from
\cite[Theorem~5.2(i)]{ipa}
that $\bm{y}_n$ converges weakly to some $\hat{\bm{y}}\in\fix\TT$, and
\ref{teogenquasiconv} is obtained from \eqref{e:relfix}.

On the other hand, if we assume \ref{teogenquasiquasihyp2}, the results follow from 
\cite[Proposition~5.16]{1} and \eqref{e:relfix}.  
\end{proof}

\begin{rem}
\begin{enumerate}
\item Previous results does not include summable errors for ease of 
the presentation, 
but they can be included effortlessly.

\item In the case when $\vv$ is strongly monotone, we have $\rvv 
= \hh$, $P_{\rvv}=\Id$, and Proposition~\ref{teogenquasi}(i) and
Proposition~\ref{teogenquasi}(ii) 
are equivalent to
\cite[Theorem~5.2(i)]{ipa} and 
\cite[Proposition 5.16]{1}, respectively.

\item In 
\cite{combvu,Davisconvrate}, 
a version of \cite[Proposition~5.16]{1} allowing for 
operators $(\Ss_k)_{k\in\N}$ and $(\vv_k)_{k\in\N}$ varying among 
iterations is proposed.
This modification allows to include variable 
step-sizes in primal-dual algorithms. In our context, the difficulty of 
including such generalization lies on the variation of the 
real Hilbert spaces $(\ran\vv_k,\scal{\cdot}{\cdot}_{\vv_k})_{k\in\N}$, 
which complicates the asymptotic analysis.
\end{enumerate}
\end{rem}

\section{Application to Primal-Dual algorithms for monotone 
inclusions}
\label{sec:primaldual}
Now we focus on the asymptotic analysis of the relaxed primal-dual 
algorithm in \eqref{primaldualalgorithm}
for solving Problem~\ref{prob2}.
First, note that $\bm{Z}=\zer\MM$, where
\begin{equation} 
\label{def:M}
{\MM}\colon \hh \rightarrow 2^{\hh}:(x,u) \mapsto \{(y,v) \in \hh \mid y 
\in Ax + L^* u,\ v \in 
B^{-1} u -Lx\}
\end{equation}
is maximally monotone in $\hh = \H \oplus \G$
\cite[Proposition 2.7(iii)]{skew}.
We define
\begin{equation}
\label{e:defV}
\vv: \hh \rightarrow \hh :(x,u) \mapsto \left( \Upsilon^{-1} x -L^*u,\Sigma^{-1}u 
-Lx\right),
\end{equation}
where $\Sigma: \G \to \G$ and $\Upsilon : \H \to \H$ are strongly 
monotone self-adjoint linear operators such that $\|\sqrt{\Sigma} 
L\sqrt{\Upsilon}\|\le 1$. 
In the case when, $\|\sqrt{\Sigma} L\sqrt{\Upsilon}\|< 1$, $\vv$ is 
strongly 
monotone \cite[eq. (6.15)]{vuvarmet} and the primal-dual algorithm 
is obtained by applying the proximal point algorithm (PPA) to 
the maximally monotone operator $\vv^{-1}\MM$ in the space
$(\H\times\G,\pscal{\cdot}{\cdot}_{\vv})$ 
\cite{bot2013,combvu,condat,yuan,CPockDiag,vu}.
In the case when $\|\sqrt{\Sigma} L\sqrt{\Upsilon}\|=1$, $\vv$ is no longer strongly 
monotone and $\pscal{\cdot}{\cdot}_{\vv}$ does not define an inner 
product. However, if $\rvv$ is closed, 
$(\ran\vv,\pscal{\cdot}{\cdot}_\vv)$
is a real Hilbert space 
in view of Proposition~\ref{prophilbert}, and we obtain the 
convergence of the primal-dual algorithm when $\|\sqrt{\Sigma} 
L\sqrt{\Upsilon}\|\le 1$ in this Hilbert space using 
Proposition~\ref{teogenquasi}. The following result provides 
conditions on  Problem~\ref{prob2} guaranteeing that $\rvv$ is closed.
\begin{prop} \label{Vpropiedades}
In the context of Problem \ref{prob2}, set
$\hh = \H \oplus \G$, let $\Sigma: \G \to \G$ and $\Upsilon : \H \to \H$ 
be strongly monotone self-adjoint linear bounded operators such that 
$\|\sqrt{\Sigma} L\sqrt{\Upsilon}\|\le1$, and let $\vv$ be the operator 
defined in 
\eqref{e:defV}. Then, the 
following hold:
\begin{enumerate}
\item\label{Vpropiedades:1}
$\vv$ is linear, bounded, self-adjoint, and
	$\frac{\T\s}{\T+\s}$-cocoercive, where  $\s>0$ and $\T>0$ are the 
	strongly monotone constants of $\Sigma$ and $\Upsilon$, 
	respectively.
\item\label{Vpropiedades:2} The
	 followings statements are equivalent.
	 \begin{enumerate}
	 	\item\label{vclosedi} $\ran\, \vv$ is closed in $\hh$.
	 	\item\label{vclosedii}$\ran (\Sigma^{-1} - L \Upsilon L^*)$ is 
	 	closed in $\G$.
	 	\item\label{vclosediii}$\ran(\Upsilon^{-1}-   L^* \Sigma L)$ is 
	 	closed in $\H$.
	 \end{enumerate}
\end{enumerate}	 
\end{prop}
\begin{proof} \ref{Vpropiedades:1}: It is a direct consequence
of \cite[Proposition~2.1]{SDR}.
\ref{Vpropiedades:2}: 
(\ref{vclosedi} $\Rightarrow$  \ref{vclosedii}). 
Let 
$(v_n)_{n \in \N}$ be 
sequence in $\ran (\Sigma^{-1} - L \Upsilon L^*)$ such that $v_n \rightarrow 
v.$ 
Therefore, for each $n \in 
\N,$ there exists $u_n \in \G$ such that $v_n=\Sigma^{-1} u_n -  L \Upsilon L^*
u_n$. 
Note that $\vv(\Upsilon L^* 
u_n ,  u_n)=(0, v_n)\rightarrow (0, v)$. Since $\ran\, \vv$ is 
closed, 
there exists some $(x,u) \in 
\H \times \G $ such that $\vv(x,u)=(0, v)$, i.e.,
\begin{align*}
	\vv(x,u)=(0,v) \quad&\Leftrightarrow\quad  
	\begin{cases} \Upsilon^{-1}x- 
	L^* u=0\\
		\Sigma^{-1} u- L x =  v 
	\end{cases}\\
	&\Rightarrow\quad  \Sigma^{-1} u-  L \Upsilon L^* u =  v.
\end{align*} 
Then $ v \in \ran (\Sigma^{-1} -  L \Upsilon L^*)$ and, therefore, $\ran (\Sigma^{-1}- L \Upsilon L^*)$ is closed.  

(\ref{vclosedii} $\Rightarrow$  \ref{vclosedi}). 
Let 
$\big((y_n,v_n)\big)_{n \in \N}$ be a 
sequence in 
$\ran\, \vv$ such that 
$(y_n,v_n) \rightarrow (y,v).$ Then, for every $n \in \N$, there 
exists 
$(x_n,u_n)$ such that 
$(y_n,u_n)=\vv(x_n,u_n)$, or equivalently,
\begin{equation}
	\label{unoydos_2}
	\begin{cases}
		 y_n=\Upsilon^{-1}x_n- L^*u_n\\
		 v_n=\Sigma^{-1} u_n- Lx_n.
	\end{cases}
\end{equation}
By applying $L \Upsilon $ to the first equation in \eqref{unoydos_2} and  
adding it to the second equation, by the continuity of $\Upsilon$ and $L$, we obtain 
\begin{equation}
	(\Sigma^{-1} -  L \Upsilon L^*)u_n= L \Upsilon y_n +  v_n\:\to\:  L  \Upsilon y +  v.
\end{equation}
Hence, since $\ran (\Sigma^{-1} - L \Upsilon L^*)$ is closed, 
there exists $u \in \G$ such that $L \Upsilon y+ v =(\Sigma^{-1} - L \Upsilon L^*)u$. 
We deduce $\vv\left( \Upsilon 
(L^*u+y),u \right)=(y,v)$ and, therefore, $\ran\, \vv$ is closed. 

(\ref{vclosedi} $\Leftrightarrow$  \ref{vclosediii}). Define 
$\tilde{\VV} : 
\G\oplus \H \to \G \oplus \H : (u,x) \mapsto (\Sigma^{-1} u-Lx,\Upsilon^{-1} x-L^*u)$.  By the equivalence \ref{vclosedi} $\Leftrightarrow$  \ref{vclosedii}
$\ran \tilde{\VV}$ is closed if and only if $\ran (\Upsilon^{-1} - L^* \Sigma L)$ is closed. 
Consider the isometric map $\boldsymbol{\Lambda} : \H \oplus \G 
\to 
\G \oplus \H : (x,u) 
\mapsto (u,x)$. Since $\boldsymbol{\Lambda}\circ 
\VV=\tilde{\VV}$, 
$\rvv$ is closed 
if 
and only if $\ran \tilde{\VV}$ is closed and the result follows.
  
\end{proof}

\begin{rem}
\label{r:closed}
\begin{enumerate}
\item 
\label{r:closed1}
In the case when $\|\sqrt{\Sigma} L\sqrt{\Upsilon}\|<1$,  we have that 
$\Upsilon^{-1} - L\Sigma L^*$ is strongly  monotone and, thus, 
invertible. This is indeed an equivalence which follows from 
\cite[eq. (2.7)]{SDR}.
Therefore, $\ran (\Upsilon^{-1} - L\Sigma L^*)=\G$ and 
$\ran\, \vv$ is closed in view of Proposition~\ref{Vpropiedades}.

\item 
\label{r:closed2}
Assume that $\ran L = \G$. Note that, for every $u \in \G$, 
$\scal{L\Upsilon L^* u}{u} \geq \T \|L^*u\|^2 \geq \T \alpha^2\|u\|^2$, 
where $\T > 0$ is the strong monotonicity parameter of $\Upsilon$ 
and the existence
of $\alpha > 0$ is guaranteed by \cite[Fact 2.26]{1}. Hence, by setting $\Sigma = (L\Upsilon L^*)^{-1}$ we have
$\Sigma^{-1}-  L \Upsilon L^*=0$. Hence, $\ran(\Sigma^{-1}-  L 
\Upsilon L^*)=\{0\}$ which is closed and 
Proposition~\ref{Vpropiedades} 
implies that $\ran\, \vv$ is closed. 
This case arises in wavelets transformations in image and signal 
processing (see, e.g., \cite{wavelet}).
\end{enumerate}
\end{rem}

The next theorem is the main result of this section, in 
which we interpret the primal-dual splitting as a relaxed proximal point 
algorithm (PPA)
applied to the primal-dual operator 
\begin{equation}
\label{e:defW}
\WW\colon \hh \rightarrow 2^{\hh}\colon(x,u)\mapsto\{(y,v)\in \hh \mid 
\vv(y,v) \in 
{\MM}(x,u)\},
\end{equation}
where $\MM$ and $\vv$ are defined in \eqref{def:M} and 
\eqref{e:defV}, respectively.
Note that, in the case when $\|\sqrt{\Sigma} L\sqrt{\Upsilon}\|<1$, 
$\vv$ is invertible and
$\WW=\vv^{-1}\MM$, which is maximally monotone in 
$(\HH,\pscal{\cdot}{\vv\cdot})$ in view of 
\cite[Proposition~20.24]{1}. This implies that $J_{\WW}$
is firmly nonexpansive under the same metric 
\cite[Proposition~23.8(iii)]{1}. These properties do not hold when 
$\|\sqrt{\Sigma} L\sqrt{\Upsilon}\|=1$, but $P_{\rvv}\circ J_{\WW}$ is 
firmly 
nonexpansive in the real Hilbert space $(\rvv,\pscal{\cdot}{\vv\cdot})$,
from which the weak convergence of primal-dual algorithm is obtained.

\begin{teo} \label{teo=1}
	In the context of Problem~\ref{prob2}, let $\vv$ be the 
	operator defined in 
	\eqref{e:defV}, where $\Sigma: \G \to \G$ and $\Upsilon : \H \to \H$ 
	are self-adjoint linear strongly monotone operators such that 
	$\|\sqrt{\Sigma} L\sqrt{\Upsilon}\|\le 1$, and suppose that $\ran\, 
	\vv$ is closed. Moreover, 
	let 
	$(\lambda_n)_{n\in\N}$ be a sequence in 
	$\left[0,2\right]$ 
	satisfying $\sum_{n\in\N}\lambda_n(2-\lambda_n)=\pinf$,  and 
consider the sequence $\big((x_n,u_n)\big)_{n \in \N}$ defined by the 
recurrence
\begin{equation}\label{ALG}
(\forall n\in\N)\quad 
	\begin{array}{l}
	\left\lfloor
	\begin{array}{l}
p_{n+1}=J_{\Upsilon A}(x_n-\Upsilon L^*u_n)\\
q_{n+1}=J_{\Sigma B^{-1}}\left( u_n+\Sigma L(2p_{n+1}-x_n)\right)\\
(x_{n+1},u_{n+1})=(1-\lambda_n)(x_n,u_n)+\lambda_n(p_{n+1},q_{n+1}),
\end{array}
\right.
\end{array}
\end{equation}
where $(x_0,u_0)\in \H\times\G$.
Then $\big(P_{\rvv} (x_n,u_n)\big)_{n \in \N}$ converges weakly in 
$(\ran\, \vv,
\pscal{\cdot }{\cdot}_{\vv})$ to some $ (\hat{y},\hat{v}) \in \fix 
(P_{\rvv}\circ J_{\WW})$, where $\WW$ is defined in \eqref{e:defW}. 
Moreover,
\begin{equation}
\big(J_{\Upsilon A}(\hat{y}-\Upsilon L^*\hat{v}),
J_{\Sigma B^{-1}}\left(\hat{v}+\Sigma L(2J_{\Upsilon A}(\hat{y}-\Upsilon L^*\hat{v})-\hat{y})\right)\big)
\end{equation}
is a solution to Problem \ref{prob2}.
\end{teo}
\begin{proof} 
First, it follows from 
Proposition~\ref{Vpropiedades}\eqref{Vpropiedades:1} that 
$\vv$ is a monotone self-adjoint linear bounded operator. Note that
\begin{equation}
\label{e:fixrel}
\fix J_{\WW}=\zer{\WW}=\zer{\MM}=\bm{Z}\neq\varnothing
\end{equation}
and, for every $(x,u)$ and $(p,q)$ in 
$\hh$,
\begin{align}
\label{e:JW}
(p,q)\in J_{\WW}(x,u)\quad  
&\Leftrightarrow\quad   (x-p,u-q) \in {\WW}(p,q) \nonumber\\
&\Leftrightarrow\quad   \vv(x-p,u-q) \in {\MM}(p,q)\nonumber\\
& \Leftrightarrow \quad \begin{cases}
\Upsilon^{-1} (x-p)-L^*(u-q) \in Ap + L^*q,\\
\Sigma^{-1}(u-q)- L(x-p) \in B^{-1}q-Lp.
\end{cases}\nonumber\\
& \Leftrightarrow \quad \begin{cases}
p = J_{\Upsilon A}(x-\Upsilon L^*u),\\
q=J_{\Sigma B^{-1}} \left( u+\Sigma 
L(2p-x)\right).
\end{cases}
\end{align}
Hence, $J_{\WW}$ is single valued
and,  for every $(x,u)\in\hh$, 
\begin{align*}
J_{\WW}(x,u)
&=\big(J_{\Upsilon A} \left( x-\Upsilon  L^*u ),J_{\Sigma B^{-1}}( u-\Sigma Lx+2\Sigma L J_{\Upsilon  
A}\left(x-\Upsilon  L^*u \right) 
\right) \big)\\
&=\bm{R}\left(\Upsilon^{-1}x-L^*u,\Sigma^{-1} u  -Lx\right)\\
&=\bm{R}(\vv(x,u)),
 \end{align*}
where $\bm{R}\colon(x,u) \mapsto \big(J_{\Upsilon  A} (\Upsilon  x), J_{\Sigma B^{-1}}( \Sigma 
u + 2\Sigma L J_{\Upsilon  	A}(\Upsilon  x) ) \big)$, which yields
\begin{equation}
\label{e:SesSP}
J_{\WW}=\bm{R}\circ \vv=\bm{R}\circ \vv\circ 
P_{\ran\vv}=J_{\WW}\circ P_{\ran\vv}.
\end{equation} 
Moreover, by defining $\TT=P_{\rvv}\circ J_{\WW}$, we deduce 
from $\vv=\vv\circ P_{\rvv}$, \eqref{e:defW}, $\ker\vv\oplus\rvv=\HH$, 
and the monotonicity of 
$\bm{M}$ that, for every $\bm{z}$ and $\bm{w}$ in 
$\rvv$, 
\begin{align}
&\pscal{\TT\bm{z}-\TT\bm{w}}
{(\Id-\TT)\bm{z}-(\Id-\TT)\bm{w}}_{\vv} \nonumber\\
&\hspace{3cm}=\pscal{J_{\WW}\bm{z}-J_{\WW}\bm{w}}
{\vv(\bm{z}-J_{\WW}\bm{z})-\vv(\bm{w}-J_{\WW}\bm{w})}\nonumber\\
&\hspace{3cm}\geq 0,
\end{align}
which yields the firm nonexpansivity of $\TT$ in 
$(\rvv,\pscal{\cdot}{\vv\cdot})$.
Therefore, it follows from \eqref{e:fixrel} and \eqref{e:SesSP} 
that Problem~\ref{prob2} is a particular instance of 
Problem~\ref{prob1} with $\Ss=J_{\WW}$. In addition,
\eqref{ALG} and \eqref{e:JW} yields
\begin{equation}
(\forall n\in\N)\quad \x_{n+1}=(1-\lambda_n)\x_n+\lambda_n 
J_{\WW}\x_n,
\end{equation}
where, for every $n\in\N$,
$\x_n=(x_n,u_n)$. 
Altogether, we obtain the results by applying 
Theorem~\ref{teogenquasi}(ii)
with $\alpha=1/2$ and $\Ss=J_{\WW}$.
  \end{proof}

\begin{rem}
	\label{r:optim}
	\begin{enumerate}
\item Since $J_{\WW}\circ P_{\rvv} = J_{\WW}$, the 
sequence $\big(P_{\ran\vv}(x_n,u_n)\big)_{n\in\N}$ is not needed in 
practice. Indeed, since
\begin{equation*}
(\forall \x \in \hh)\quad \|\x\|_{\vv}=\|P_{\rvv} \x\|_{\vv},
\end{equation*}
we can use a stopping criteria only involving $\big((x_n,u_n)\big)_{n 
\in \mathbb{N}}.$

\item\label{r:optim2} Suppose that $\G=\oplus_{i=1}^m\G_i$, 
$B\colon (u_i)_{1\le i\le m}\mapsto \times_{i=1}^m B_iu_i$,
$\Sigma\colon (u_i)_{1\le i\le m}\mapsto (\Sigma_iu_i)_{1\le i\le m}$,
and $L\colon x\mapsto (L_ix)_{1\le i\le m}$,
where, for every $i\in\{1,\ldots,m\}$, $\G_i$ is a real Hilbert space, 
$B_i$ is maximally monotone, $\Sigma_i\colon \G_i\to\G_i$ is a 
strongly monotone self-adjoint linear bounded operator,
and $L_i\colon\H\to\G_i$ is a linear bounded operator. In this 
context, the inclusion in \eqref{e:priminc} is equivalent to
\begin{equation}
\label{e:vured}
\text{find}\quad x\in\H\quad\text{such that}\quad  0 \in A 
x+\sum_{i=1}^{m} L_{i}^{*}B_{i} 
L_{i} x.
\end{equation}
Then, under the assumptions 
\begin{equation} \label{e:CondialgPDSum}
\sum_{i=1}^m\left\|\sqrt{\Sigma_i}L_i\sqrt{\Upsilon}\right\|^2\le 1\quad 
\text{and}\quad 
\ran\left(\Upsilon^{-1}-\sum_{i=1}^mL_i^*\Sigma_iL_i\right)\quad
\text{is closed,}
\end{equation}
Proposition~\ref{Vpropiedades} and Theorem~\ref{teo=1} ensures 
the convergence 
of \eqref{ALG}, which reduces to 
\begin{equation}\label{e:algPDSum}
	(\forall n\in\N)\quad 
	\begin{array}{l}
	\left\lfloor
	\begin{array}{l}
{p}_{n+1}= J_{\varUpsilon A} ({x}_n -\varUpsilon \sum_{i=1}^mL_i^* 
{u}_{i,n})\\
x_{n+1}=(1-\lambda_n)x_n+\lambda_n p_{n+1}\\
\text{for } i = 1, \ldots, m  \\
\left\lfloor
	\begin{array}{l}
	{q}_{i,n+1}=J_{\Sigma_i B_i^{-1}}({u}_{i,n}+\Sigma_i
L_i(2{p}_{n+1}-{x}_n))\\
	u_{i,n+1}=(1-\lambda_n)u_{i,n}+\lambda_n q_{i,n+1}.
	\end{array}\right.
\end{array}
\right.
\end{array}
\end{equation}
Note that \eqref{e:algPDSum} has the same structure than the 
algorithm in \cite[Corolary 6.2]{vuvarmet} without considering 
cocoercive operators and the convergence is guaranteed under the 
weaker assumption \eqref{e:CondialgPDSum} in view of 
Remark~\ref{r:closed}\eqref{r:closed1}.

\item\label{r:optimi} In the context of 
the optimization problem in \eqref{e:optim}, \eqref{ALG} 
reduces to
\begin{equation}\label{ALGoptim}
(\forall n\in\N)\quad 
	\begin{array}{l}
	\left\lfloor
	\begin{array}{l}
p_{n+1}=\prox_{f}^{\Upsilon^{-1}} (x_n-\Upsilon L^*u_n)\\
q_{n+1}=\prox_{ g^*}^{\Sigma^{-1}} \left( u_n+\Sigma L(2p_{n+1}-x_n)\right)\\
(x_{n+1},u_{n+1})=(1-\lambda_n)(x_n,u_n)+\lambda_n(p_{n+1},q_{n+1}),	
\end{array}
\right.
\end{array}
\end{equation}
Under the additional condition 
$\ran\vv$ closed, Theorem~\ref{teo=1} generalizes 
\cite[Theorem~3.3]{condat} to infinite 
dimensional spaces and allowing preconditioners and a larger choice 
of parameters 
$(\lambda_n)_{n\in\N}$. 
Indeed,  in finite 
dimensional spaces, $\ran\, \vv$ is closed, 
Theorem~\ref{teo=1} implies that 
$P_{\rvv} (x_n,u_n)\to (\hat{y},\hat{v}) \in \rvv$ in $(\ran\, \vv,
\pscal{\cdot }{\cdot}_{\vv})$ and, since $J_{\WW}=J_{\WW}\circ 
P_{\rvv}$ is continuous, we conclude
 $(p_{n+1},q_{n+1})=J_{\WW}(x_n,u_n)=J_{\WW}(P_{\rvv} 
 (x_n,u_n))\to
J_{\WW}(\hat{y},\hat{v})\in \bm{Z}$. In order to 
guarantee the convergence of the relaxed sequence 
$((x_n,u_n))_{n\in\N}$, it is enough to suppose 
$(\lambda_n)_{n\in\N}\subset\left[\epsilon, 2-\epsilon\right]$
for some $\epsilon\in\left]0,1\right[$, and use the argument 
in \cite[p.473]{condat}.

\item In the particular case when $ \|\sqrt{\Sigma}L\sqrt{\Upsilon}\| < 
1$, it follows from \cite[eq. (6.15)]{combvu} (see also 
\cite[Lemma~1]{CPockDiag}) that 
$\vv$ is strongly monotone, which yields $\rvv=\hh$ and 
$P_{\rvv}=\Id$. Hence, we recover from Theorem~\ref{teo=1} the
weak convergence of $((x_n,u_n))_{n\in\N}$ to a solution to 
Problem~\ref{prob2} proved 
in \cite{bot2013,combvu,condat,yuan,CPockDiag,vu}. 

\item In the particular instance when $\lambda_n\equiv1$, 
the weak convergence of \eqref{ALG} is deduced without any range 
closedness in \cite[Remark~3.4(4)]{SDR}. The result is obtained from 
an alternative formulation of the algorithm and the extension to 
$\lambda_n\not\equiv1$ is not clear. As we will show in 
Section~\ref{sec:TV}, the additional relaxation step is relevant in the 
efficiency of the algorithm.
\end{enumerate}
\end{rem}


\subsection{Case $L=\id$: Douglas--Rachford splitting}
\label{sec:DR}
In this section, we study the particular case 
of Problem~\ref{prob2} when $L=\id$. In this context,
the following result is a refinement of Theorem~\ref{teo=1},
which relates 
the primal-dual algorithm in \eqref{ALG} with 
Douglas-Rachford splitting (DRS) when
 \begin{equation}
\label{e:tausigmaDR}
 \Upsilon=\Sigma^{-1} \text{ is strongly monotone}.
 \end{equation}
When $\Upsilon=\tau\id$ and $\Sigma=\s\id$, \eqref{e:tausigmaDR} 
reads $\sigma\tau=1$ and the connection of \eqref{ALG} with DRS
is discovered in \cite[Section~4.2]{cp} in the optimization 
context. However, the convergence is guaranteed only if $\T\s<1$, 
which is extended to the case $\sigma\tau=1$ in 
\cite[Section~3.1.3]{condat} in the finite dimensional setting.
Previous connection allows us to recover the classical convergence 
results 
in \cite{eckstein-bertsekas,lions-mercier79} when $\Upsilon = \T \id$ 
with our approach. Define the operator \begin{equation}
\label{e:opDRS}
G_{\Upsilon,B,A}=J_{\Upsilon B}\circ(2J_{\Upsilon A} -\id)+(\id-J_{\Upsilon A}),
\end{equation}
and we recall that relaxed DRS iterations are defined by the 
recurrence
\begin{equation}
z_0\in\H,\quad (\forall n\in\N)\quad 
z_{n+1}=(1-\lambda_n)z_n+\lambda_nG_{\Upsilon ,B,A}z_n,
\end{equation} 
where $(\lambda_n)_{n\in\N}$ is a sequence in 
$\left[0,2\right]$.

\begin{prop}\label{teo:DRour}
In the context of Problem~\ref{prob2}, set $L=\id$, let $\Upsilon$ be a 
strongly monotone self-adjoint linear operator,
let $(\lambda_n)_{n\in\N}$ be a sequence in 
$\left[0,2\right]$ 
satisfying $\sum_{n\in\N}\lambda_n(2-\lambda_n)=\pinf$,  and 
consider the sequence $\big((x_n,u_n)\big)_{n \in \N}$ defined by the 
recurrence
\begin{equation}\label{ALGDR}
(\forall n\in\N)\quad 
	\begin{array}{l}
	\left\lfloor
	\begin{array}{l}
	p_{n+1}=J_{\Upsilon A}(x_n-\Upsilon u_n)\\
q_{n+1}=J_{\Upsilon^{-1}B^{-1}}\big( u_n+\Upsilon^{-1}(2p_{n+1}-x_n)\big)\\
(x_{n+1},u_{n+1})=(1-\lambda_n)(x_n,u_n)+\lambda_n(p_{n+1},q_{n+1}),
\end{array}
\right.
\end{array}
\end{equation}
where $(x_0,u_0)\in \H\times\H$. 
Then, by setting, for every $n\in\N$, $z_n=x_n-\Upsilon y_n$,  
$(z_n)_{n \in \N}$ 
converges 
weakly in $\H$ to some
$\hat{z}\in \fix G_{\Upsilon,B,A}$ and 
$$\left(J_{\Upsilon A} \hat{z}, 
-\Upsilon^{-1}(\hat{z}-J_{\Upsilon A}\hat{z})\right)$$ is a solution to 
Problem \ref{prob2}.
Moreover, we have
\begin{equation}
\label{e:algDRclassic}
(\forall n\in\N)\quad 
z_{n+1}=(1-\lambda_n)z_n+\lambda_nG_{\Upsilon,B,A}z_n.
\end{equation}
\end{prop}
\begin{proof}
Note that, since $L=\id$, \eqref{e:defV} and \eqref{e:tausigmaDR} 
yield $\vv\colon (x,u)\mapsto (\Upsilon^{-1}x-u,\Upsilon u-x)$ and
Remark~\ref{r:closed}\eqref{r:closed2} 
implies that
$\ran\, \vv$ is closed. Hence, it follows from 
Theorem~\ref{teo=1} and \eqref{e:JW}
in the case $L=\id$
that $(P_{\ran\,\vv}(x_n,u_n))_{n\in\N}$
converges weakly in 
$(\rvv,\pscal{\cdot}{\cdot}_{\vv})$ to some 
$(\hat{y},\hat{v})\in\fix (P_{\rvv}\circ J_{\WW})$
and 
\begin{equation}
\label{e:solhat}
(\hat{x},\hat{u})=J_{\WW}(\hat{y},\hat{v})=\big(J_{\Upsilon 
A}(\hat{y}-\Upsilon\hat{v}),
J_{\Upsilon^{-1} B^{-1}}\left(\hat{v}+\Upsilon^{-1}(2\hat{x}-\hat{y})\right)\big)\in\bm{Z}.
\end{equation}
Set $\Lambda\colon  (x,u)\mapsto x-\Upsilon u$, and $ \boldsymbol{\Upsilon} \colon (x,u) \mapsto (\Upsilon x, \Upsilon u)$.
Note that $\Lambda$ is surjective, that 
\begin{equation}
\label{e:propLamb}
\Lambda^*   \Lambda = \vv \circ \boldsymbol{\Upsilon},\quad 
\ran \vv =\ran\Lambda^*,
\end{equation}
and, in view of 
\cite[Fact~2.25(iv)]{1}, that
\begin{equation}
\label{e:descdsum}
\H\times\H=\ran\Lambda^*\oplus\ker\Lambda.
\end{equation}
Then, for every $(x,u)\in\H\times\H$, it follows from 
\eqref{e:defW} and \eqref{e:JW} in the case $L=\id$, 
\eqref{e:tausigmaDR},
\cite[Proposition~23.34 (iii)]{1}, and \eqref{e:opDRS} that
\begin{align}
\label{e:relopsDR}
\Lambda(J_{{\WW}}(x,u))&=J_{\Upsilon  A} \left( x-\Upsilon u )-\Upsilon J_{\Upsilon^{-1} B^{-1}}\Upsilon^{-1}( 
\Upsilon u-  x+2 J_{\Upsilon
A}\left(x-\Upsilon u \right)
\right)\nonumber\\
&= -J_{\Upsilon A } (\Lambda(x,u)) + \Lambda(x,u) + J_{\Upsilon B}(2J_{\Upsilon A } 
(\Lambda(x,u))-\Lambda(x,u))\nonumber\\
&=G_{\Upsilon,B,A}(\Lambda(x,u)).
\end{align}
Moreover, since $(\hat{y},\hat{v})\in
\fix(P_{\rvv}\circ J_{\WW})$, by using 
\eqref{e:relopsDR}, \eqref{e:descdsum},
 and \eqref{e:propLamb} we deduce 
\begin{align}
\label{e:incfixpoints}
G_{\Upsilon,B,A}(\Lambda(\hat{y},\hat{v}))&=\Lambda 
(J_{\WW}(\hat{y},\hat{v}))\nonumber\\
&=\Lambda\circ 
P_{\ran\,\Lambda^*}(J_{\WW}(\hat{y},\hat{v}))\nonumber\\
&=\Lambda(P_{\rvv}\circ J_{\WW}(\hat{y},\hat{v}))\nonumber\\
&=\Lambda(\hat{y},\hat{v}),
\end{align}
and, thus, defining $\hat{z}=\Lambda(\hat{y},\hat{v})$, we obtain
$\hat{z}\in \fix G_{\Upsilon,B,A}$. In addition, it follows from 
 \eqref{e:solhat} that  $\hat{x}=J_{\Upsilon A}\hat{z}$ and, since
$(\hat{y},\hat{v})\in\fix (P_{\rvv}\circ J_{\WW})$,
we deduce from \eqref{e:descdsum} and \eqref{e:JW} that
\begin{equation}
\hat{z}=\Lambda(\hat{y},\hat{v})=\Lambda \left(P_{\rvv}\circ 
J_{\WW}(\hat{y},\hat{v})\right)
 =\Lambda 
 J_{\WW}(\hat{y},\hat{v})=\Lambda(\hat{x},\hat{u})=\hat{x}-\Upsilon\hat{u},
\end{equation}
 which yields
 $\hat{u}=-\Upsilon^{-1}(\hat{z}-J_{\Upsilon
 		A}\hat{z})$.
Furthermore, noting that, for every $n\in\N$, 
$z_n=\Lambda(x_n,u_n)$, we deduce from \eqref{ALGDR},
\eqref{e:JW}, and \eqref{e:relopsDR} that
\begin{align}
(\forall n\in\N)\quad 
z_{n+1}&=\Lambda(x_{n+1},u_{n+1})\nonumber\\
&=(1-\lambda_n)\Lambda(x_n,u_n)+\lambda_n 
\Lambda \left(J_{\WW}(x_n,u_n)\right)\nonumber\\
&=(1-\lambda_n)z_n+\lambda_n G_{\Upsilon ,B,A}z_n.
\end{align}
Finally, in order to prove the weak convergence of $(z_n)_{n\in\N}$ to 
$\hat{z}$, fix $w\in\H$ and set $(p,q)=((\id+\Upsilon^2)^{-1}w,-\Upsilon (\id+\Upsilon^2)^{-1} w)$. We have $(p,q)\in\ran\,\Lambda^*$, 
$\Lambda(p,q)=w$ and it follows from \eqref{e:descdsum},
\eqref{e:propLamb}, $(\hat{y},\hat{v})\in\rvv=\ran \Lambda^*$, and $(P_{\ran\,\vv}(x_n,u_n))_{n\in\N} \weakly (\hat{y},\hat{v})$ that
\begin{align}
\scal{z_n-\hat{z}}{w}&=\scal{\Lambda(x_n-\hat{y},u_n-\hat{v})}{\Lambda(p,q)}\nonumber\\
&=\scal{\Lambda\, 
P_{\ran\,\Lambda^*}(x_n-\hat{y},u_n-\hat{v})}{\Lambda(p,q)}\nonumber\\
&=\pscal{P_{\ran\,\Lambda^*}(x_n-\hat{y},u_n-\hat{v})}{\vv  (\Upsilon p, \Upsilon q)}\nonumber\\
&=\pscal{P_{\ran\,\Lambda^*}(x_n-\hat{y},u_n-\hat{v})}{(\Upsilon p, \Upsilon q)}_{\vv}\nonumber\\
&=\pscal{P_{\ran\,\Lambda^*}(x_n,u_n)-(\hat{y},\hat{v})}{(\Upsilon p, \Upsilon q)}_{\vv}\to
 0
\end{align}
and the result follows.
  
\end{proof}

\begin{rem}
\begin{enumerate}
\item From the proof of Proposition~\ref{teo:DRour}, we 
deduce
$\Lambda(\fix (P_{\rvv}\circ J_{\WW}))\subset \fix G_{\Upsilon,B,A}$. 
The 
converse inclusion 
is also true, as detailed in Proposition~\ref{prop:propiedad} in 
the Appendix. 

\item Proposition~\ref{teo:DRour} provides a connection between 
classical 
Douglas--Rachford scheme 
\cite{eckstein-bertsekas} and the primal-dual version in 
\eqref{ALGDR}, and we obtain 
that the auxiliary sequence $(z_n)_{n\in\N}$ converges weakly to a 
$\hat{z}$ whose 
\textit{primal-dual
shadow} is a primal-dual solution. In \cite{svaiter} (see also 
\cite{BauschkeFirmly,BausMoursi}) the weak 
convergence of the primal-dual
shadow sequence is proved in the case $\lambda_n\equiv 1$, 
by reformulating DRS as an alternative algorithm with
 primal-dual iterates in $\gra A$. This technique does not allow 
for relaxation steps, since after relaxation the iterates are no longer 
in $\gra A$ unless it is affine linear. 
\end{enumerate}
\end{rem}
 
\section{Numerical experiments}
\label{sec:TV}
A classical model in image processing is the total variation 
image restoration \cite{ROF}, which aims at 
recovering an image from a blurred and noisy observation under
 piecewise constant assumption on the solution.
The model is formulated via the optimization problem
\begin{equation}\label{pro:TV}
\min_{x \in [0,255]^{N}} 
\frac{1}{2}\|Rx-b\|^2_{2}+\alpha\|\nabla x\|_1=:F^{TV}(x),
\end{equation} 
where $x\in [0,255]^{N}$ is the image of $N=N_1 \times N_2$ pixels 
to recover from a 
blurred and noisy observation $b\in \R^{m}$, $R : 
\R^{N} \rightarrow \R^{m}$ is a linear operator representing a 
Gaussian blur, the 
discrete gradient
$\nabla\colon x\mapsto (D_1x,D_2x)$ includes horizontal 
and 
vertical differences through linear operators $D_1\colon\R^N\to\R^N$ 
and $D_2\colon\R^N\to\R^N$, 
respectively, 
its adjoint $\nabla^*$ is the discrete divergence
(see, e.g., \cite{TV-chambolle}), and $\alpha \in \RPP$. 
 A difficulty in this model is 
the presence of the non-smooth $\ell^1$ 
norm composed with the discrete 
gradient operator $\nabla$, which is non-differentiable and 
its proximity operator has not a closed form. 

Note that, by setting $f=\|R\cdot-b\|^2/2$, 
$g_1=\alpha\|\cdot\|_1=g_2$,
and $g_3=\iota_{[0,255]^N}$, $L_1=D_1$, $L_2=D_2$, 
and $L_3=\id$,
 \eqref{pro:TV} can be reformulated as $\min (f+\sum_{i=1}^3g_i\circ 
 L_i)$ or equivalently as (qualification condition holds)
\begin{equation}
\text{find}\quad x\in \R^N\quad\text{such that}\quad0\in\partial 
f(x)+\sum_{i=1}^3L_i^*\partial g_i(L_i x),
\end{equation}
which is a particular instance of \eqref{e:vured}, in view of 
\cite[Theorem~20.25]{1}.
Moreover, for every $\T>0$, $J_{\T \partial f}=(\id+\T 
R^*R)^{-1}(\id-\T R^*b)$, for 
every $i\in\{1,2,3\}$,
$J_{\T (\partial g_i)^{-1}}=\T(\id-\prox_{g_i/\T})(\id/\T)$, 
$\prox_{g_3/\T}=P_{[0,255]^N}$, and, for $i\in\{1,2\}$,
$\prox_{g_i/\T}=\prox_{\alpha\|\cdot\|_1/\T}$ is the component-wise 
soft thresholder, computed in 
\cite[Example~24.34]{1}. Note that $(\id+\T R^*R)^{-1}$
 can be computed efficiently via a diagonalization of 
$R$ using the fast Fourier transform $F$ 
\cite[Section 4.3]{deblur}. 
Altogether, Remark~\ref{r:optim}.\eqref{r:optim2} allows us to 
write algorithm in \eqref{e:algPDSum} as Algorithm~\ref{alg:TV} below, 
where we set $\varUpsilon=\tau\id$, $\Sigma_1=\sigma_1\id$, 
$\Sigma_2=\sigma_2\id$, and  $\Sigma_3=\sigma_3\id$, for $\T>0$, $\s_1>0$, $\s_2>0$,
and $\s_3>0$. We denote by $\mathcal{R}$  the primal-dual 
error
\begin{equation}
	\label{e:defR}
\mathcal{R} : (x_+,u_+,x,u) \mapsto 
\sqrt{\frac{\|(x_+,u_+)-(x,u)\|^2}{\|(x,u)\|^2}}
\end{equation} 
and by $\varepsilon>0$ the convergence tolerance.

\begin{algorithm}[H]
\caption{} 
\label{alg:TV}
\begin{algorithmic}[1]
\STATE{Fix $x_0, u_{1,0},u_{2,0}$, and $u_{3,0}$  in 
$\R^{N}$,  let $\tau$, $\sigma_1$, $\sigma_2$,
and $\sigma_3$ be in $\RPP$, let $(\lambda_n)_{n\in\N}$ in 
$\left[0,2\right]$ 
such that $\sum_{n\in\N}\lambda_n(2-\lambda_n)=\pinf$, and fix $r_0 
> \varepsilon>0$. }
\WHILE{ $r_n > \varepsilon $}
\STATE{$p_{n+1}= (\id+\T R^*R)^{-1}(x_n-\T (D_1^*u_{1,n} + D_2^* 
u_{2,n}+u_{3,n}+ 
R^*b))$}
\STATE{$x_{n+1}= (1-\lambda_n) x_{n} +  \lambda_n p_{n+1}$}
\STATE{$q_{1,n+1}= 
\s_1(\id-\prox_{\alpha\|\cdot\|_1/\s_1})(u_{1,n}/\s_1+D_1(2p_{n+1}-x_n))$}
\STATE{$q_{2,n+1}= 
\s_2(\id-\prox_{\alpha\|\cdot\|_1/\s_2})(u_{2,n}/\s_1+D_2(2p_{n+1}-x_n))$}
\STATE{$q_{3,n+1}=\s_3\big( \id-  P_{[0,255]^N}\big)\left( 
u_{3,n}/\s_3+ 2p_{n+1}-x_n\right)$}
\STATE{$
\begin{array}{l}
	\left\lfloor
	\begin{array}{l} \text{for } i=1,2,3\\
	u_{i,n+1}= (1-\lambda_n) u_{i,n} +\lambda_n q_{i,n+1}
	\end{array}
\right.
\end{array}
$}
\STATE{ 
$r_n=\mathcal{R}\big((x_{n+1},u_{1,n+1},u_{2,n+1},u_{3,n+1}),
(x_n,u_{1,n},u_{2,n},u_{3,n})\big)$ }
\ENDWHILE 
\RETURN{$(x_{n+1},u_{1,n+1},u_{2,n+1},u_{3,n+1})$}
\end{algorithmic}
\end{algorithm}
In this case, \eqref{e:CondialgPDSum} reduces to 
\begin{equation}
\label{e:boundary}
\T(\s_1\|D_1\|^2+\s_2\|D_2\|^2+\s_3)\le 1
\end{equation}
and the closed range condition is trivially satisfied. By using the power 
iteration 
 \cite{vonMises} with tolerance 
$10^{-9}$, we obtain 
$\|D_1\|^2=\|D_2\|^2 \approx 3.9998$.

 Observe that, when $\s_1=\s_2=\s_3=\s$, Algorithm~\ref{alg:TV} 
 reduces 
to the algorithm proposed in \cite{cp} (when $\s\T(\|D_1\|^2+\|D_2\|^2+1)<1$)
or \cite[Theorem~3.3]{condat} (algorithm denoted by \texttt{condat}), 
where the case 
$\s\T(\|D_1\|^2+\|D_2\|^2+1)=1$  is included. 

Since in \cite[Section~5.1]{SDR}, the critical step-sizes achieve the 
best performance,
we provide a numerical experiment which compare the efficiency
of Algorithm~\ref{alg:TV} for different values of the parameters $\T, 
\s_1,\s_2$, and $\s_3$ in the boundary of \eqref{e:boundary} and 
different relaxation parameters 
$\lambda_n$. In particular we compare with the case 
$\s_1=\s_2=\s_3=\s$ (\texttt{condat}), which 
turns out to be more efficient 
than other methods as AFBS \cite{Molinari2019}, MS \cite{skew}, 
Condat-V\~u \cite{condat,vu} in this context \cite[Section~5.1]{SDR}. 
For these 
comparisons, we consider the test image $\overline{x}$ shown in
Figure~\ref{fig:imreal} of $256\times 256$ pixels 
($N_1=N_2=256$) inspired in \cite[Section~5]{Yang21}. 
The 
operator blur $R$ is set as a Gaussian blur of size $9\times 9$ and 
standard 
deviation 4 (applied by MATLAB function 
\textit{fspecial}) and the observation $b$ is obtained by
$b=R\overline{x}+e\in \R^{m_1\times m_2}$, where $m_1=m_2=256$ 
and $e$ is an 
additive zero-mean white Gaussian noise with standard deviation 
$10^{-3}$ (using \textit{imnoise} function in MATLAB). We generate 
20 random realizations of the random variable $e$ leading to 20
observations $(b_i)_{1\le i\le 20}$. 

We study 
the efficiency of Algorithm~\ref{alg:TV} for different values of 
$\T,\s_1,\s_2$, and $\s_3$ and relaxation steps 
$\lambda_n\equiv\lambda \in \{1,1.5,1.9\}$. In order to approximate 
the best performant step-sizes in the boundary of \eqref{e:boundary}, 
we consider $\T \in \mathcal{C}:=\{0.10+0.05\cdot 
n\}_{n=0,\ldots,10}$  and 
$\s_1=\s_2=\s_3=\s=\T/(1+\|D_1\|^2+\|D_2\|^2)$ in the case of 
\texttt{condat}. In the case of 
Algorithm~\ref{alg:TV} we consider $\s_1=\gamma_1 
(1-\gamma_2)/(\T \|D_1\|^2) $,  $\s_2=(1-\gamma_1) 
(1-\gamma_2)/(\T \|D_2\|^2) $, $\s_3=\gamma_2/\T$, where $(\T, 
\gamma_1, \gamma_2) \in \mathcal{C}\times  \{0.01, 0.005, 0.001 
\}\times  \{0.5, 0.55, 0.6, 0.65\}$.

In Table~\ref{T:tol8franjas} we provide the average number of 
iterations
obtained by applying Algorithm~\ref{alg:TV} for solving \eqref{pro:TV} 
considering the 20 observations 
$(b_i)_{1\le i\le 20}$ and the best set of step-sizes found with 
the procedure above. The tolerance 
is set as $\varepsilon = 10^{-8}$. 
We observe that Algorithm~\ref{alg:TV} becomes more 
efficient in iterations as long as 
the relaxation parameters are larger. The case $\lambda=1.9$ 
achieves the tolerance in approximately 35\% less iterations than the 
case  $\lambda=1$. By choosing different
parameters $\s_1,\s_2$, and $\s_3$, the algorithm achieves 
the tolerance in approximately 6\% less iterations than \texttt{condat}.

 This conclusion is confirmed in
Figure~\ref{f:TV_comp}, which shows the performance 
obtained 
with the observation $b_{4}$. This figure also shows that both 
algorithms achieve in less iterations the optimal objective value for 
higher relaxation parameters, with a slight advantage of 
Algorithm~\ref{alg:TV}. Note that, since the algorithms under study 
has the same structure, the CPU time by iteration is very similar. 

In Figure~\ref{f:imagenesTV} we provide the images reconstructed 
from observation $b_4$ by using \texttt{condat} and 
Algorithm~\ref{alg:TV} 
after 
300 iterations. The best reconstruction, in terms of 
objective value $F^{TV}$ and PSNR (Peak signal-to-noise ratio) is 
obtained by Algorithm~\ref{alg:TV}.

\begin{table}
{\footnotesize	 \caption{Averages number of iterations for 
Algorithm~\ref{alg:TV} with $\T(\s_1\|D_1\|^2+\s_2\|D_1\|^2+\s_3)= 1$ 
and
 \texttt{condat} with tolerance $10^{-8}$.}\label{T:tol8franjas}
\begin{center}
\begin{tabular}{|c|c|c|c|c|c|c|}\cline{6-7}
\multicolumn{5}{c}{}  & \multicolumn{2}{|c|}{$\varepsilon=10^{-8}$} 
\\ \hline
\multicolumn{1}{|c|}{Algorithm}  & $\T$     &  $\s_{1}$ &  $\s_{2}$ & 
$\lambda_n$ & Av. Time(s) & Av. Iter. \\\hline
\multirow{9}{*}{Alg. \ref{alg:TV}} & 0.2  & 0.7425 & 0.4950 & 
\multirow{4}{*}{$1$}  & 82.6373 & 8844 \\
& 0.2  & 0.7463 & 0.4975 & & 82.2817 & 8827  \\
& 0.2  & 0.7493 & 0.4995 & & 82.5722 & 8833\\\cline{2-7}
& 0.2  & 0.7425 & 0.4950 &  \multirow{4}{*}{$1.5$}  & 63.1338 & 6766 
\\
& 0.2  & 0.7463 & 0.4975 & & 63.0645 & 6754  \\
& 0.2  & 0.7493 & 0.4995 & & 63.0996 & 6758  \\\cline{2-7}
 & 0.2  & 0.8044 & 0.4331 &  \multirow{4}{*}{$1.9$}  & 53.9059 & 
 5770  \\
 & 0.2  & 0.8085 & 0.4353 & & 53.8663 & 5767  \\
& 0.2  & 0.8117 & 0.4371 & & 53.8022 & 5761 \\ \hline
\multirow{3}{*}{\texttt{condat}} & 0.2 & - & - & 1 &  92.7997 & 9326 \\ 
& 0.2 & - & - & 1.5 & 67.0886 & 7131  \\ 
& 0.2 & - & - & 1.9 & 57.8523 &  6121  \\ 
\hline
\end{tabular}
\end{center} }
\end{table}

\begin{figure}
\centering
\subfloat[]{\includegraphics[scale=0.3]{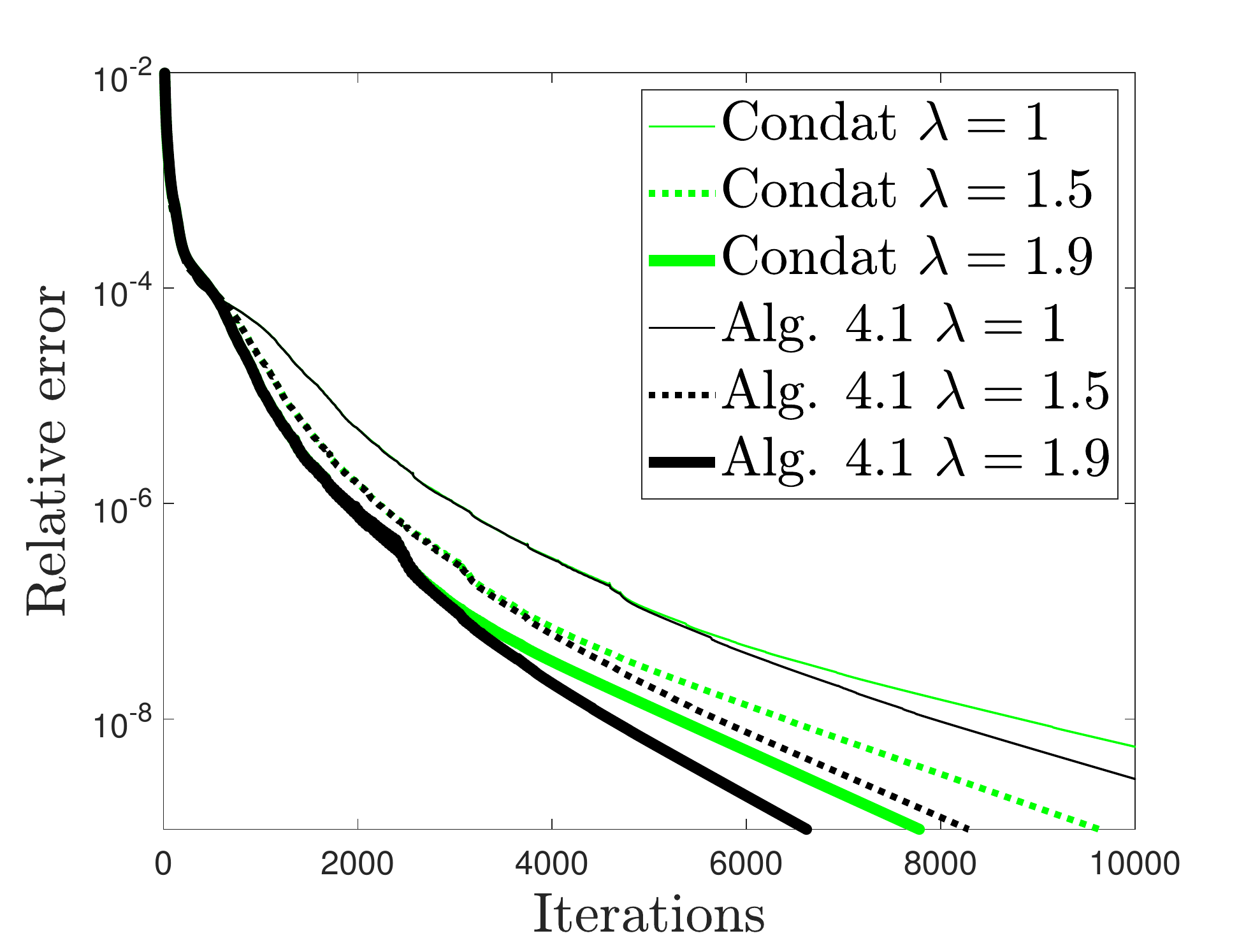}\label{f:comp2a}}
\subfloat[]{\includegraphics[scale=0.3]{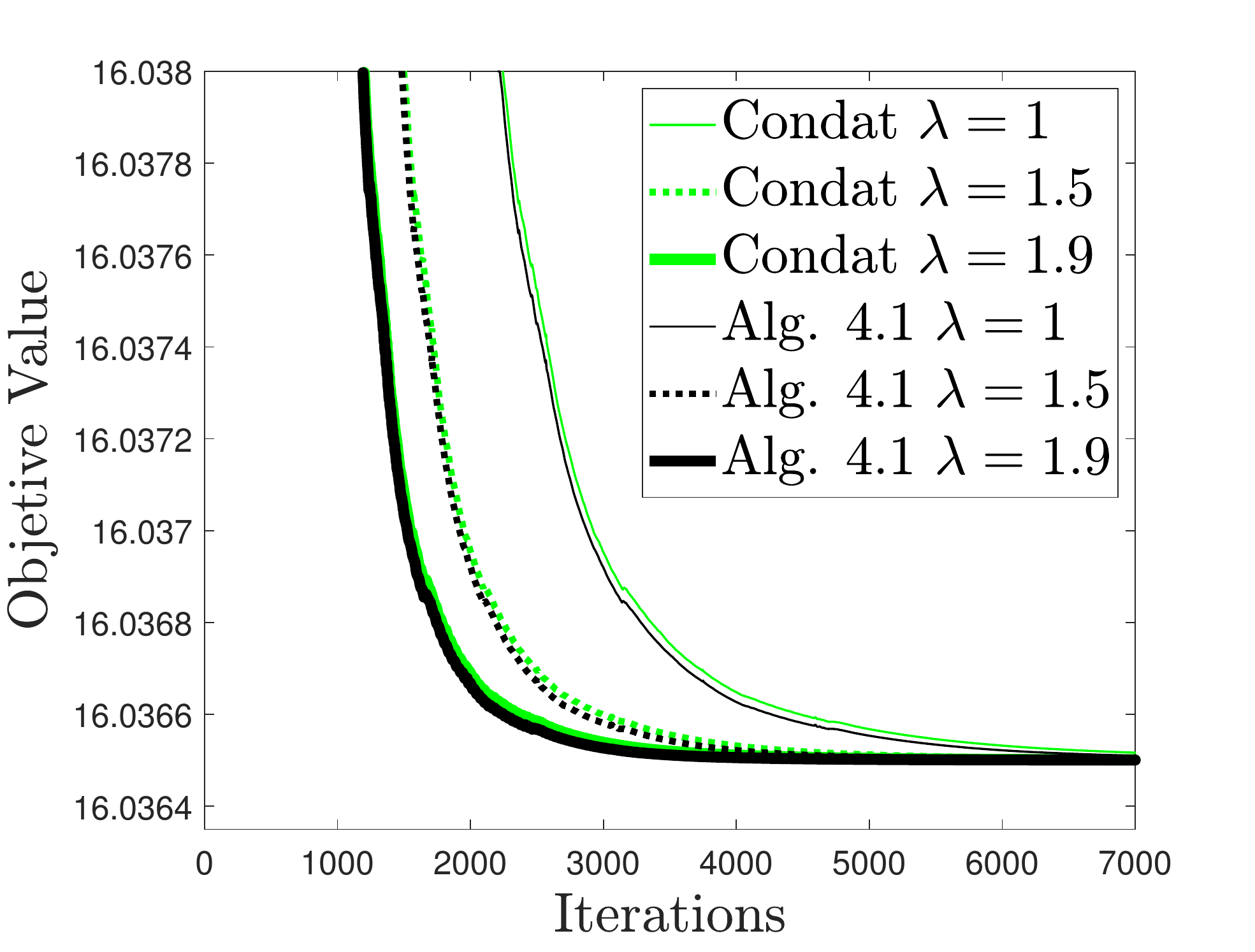}\label{f:comp2b}}
\caption{{Comparison of Algorithm~\ref{alg:TV} with 
$\T(\s_1\|D_1\|^2+\s_2\|D_1\|^2+\s_3)= 1$ and \texttt{condat} 
(observation $b_{4}$). }}
\label{f:TV_comp}
\end{figure}

\begin{figure}
	\centering
\subfloat[\scriptsize Original, $F^{TV}(\overline{x})= 17.8430$ 
]{\label{fig:imreal}\includegraphics[scale=0.5]{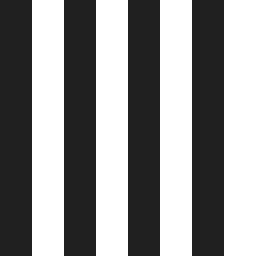}}\quad \quad
\subfloat[\scriptsize Blurry/noisy $b_{4}$, $F^{TV}(b)= 80.5807$,
PSNR=$18.4842$]{\label{fig:imblur}\includegraphics[scale=0.5]{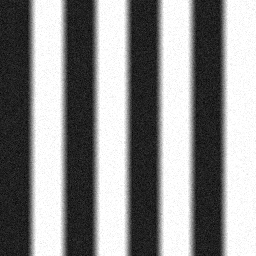}}\\
\subfloat[\scriptsize \texttt{condat} ,  
$F^{TV}(x_{300})=16.3503$, 
PSNR=$28.5957$.]{\label{fig:imAFBS}\includegraphics[scale=0.5]{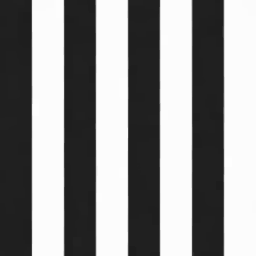}}\quad
 \quad
\subfloat[\scriptsize Alg. \ref{alg:TV}, $F^{TV}(x_{300})=16.3473$, 
PSNR=$28.6349$. 
]{\label{fig:imMS}\includegraphics[scale=0.5]{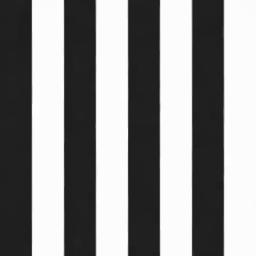}}
\caption{Reconstructed image, after 300 iterations, from blurred and 
noisy image using \texttt{condat}  and Alg. 
\ref{alg:TV} in their best cases, respectively, cases and $\lambda=1.9$ .}
\label{f:imagenesTV}
\end{figure}

\section*{Acknowledgments}
The first author thanks the support of ANID under grant FONDECYT 1190871 
and grant Redes 180032. The second author thanks the support 
of ANID-Subdirección de Capital  Humano/Doctorado 
Nacional/2018-21181024 and by the Direcci\'on de Postgrado y Programas from 
UTFSM through Programa de Incentivos a la Iniciaci\'on Cient\'ifica (PIIC).

\section{Appendix}
\label{sec:app}
\begin{prop}
	\label{prop:propiedad}
In the context of Problem~\ref{prob2}, set $L=\id$, let 
$\Upsilon\colon\H\to\H$ be a strongly monotone self adjoint linear 
bounded operator, 
set $\Lambda\colon\H\times\H\to\H\colon (x,u)\mapsto x-\Upsilon u$, 
let $\vv$, $\WW$,  and $G_{\Upsilon,B,A}$
be the operators defined in \eqref{e:defV}, \eqref{e:defW},
and \eqref{e:opDRS}, respectively. 
Then, 
$\Lambda(\fix (P_{\rvv}\circ J_{\WW}))=\fix G_{\Upsilon,B,A}$.
\end{prop}
\begin{proof} 
The inclusion $\subset$
is proved in \eqref{e:incfixpoints}. Conversely, 
since $\Lambda^*\colon z\mapsto 
(z,-\Upsilon z)$, we have $\Lambda\circ 
\Lambda^*=\id+\Upsilon^2$ and
\cite[Proposition~3.30 \& Example 3.29]{1} yields $P_{\rvv}=P_{\ran 
\Lambda^*}=\Lambda^*(\id+\Upsilon^2)^{-1}
\Lambda$. Therefore, if $\hat{z}\in \fix G_{\Upsilon,B,A}$, by setting 
$(\hat{x},\hat{u}):=\Lambda^*(\id+\Upsilon^2)^{-1}\hat{z}$, 
we have
$\hat{z}=\Lambda(\hat{x},\hat{u})$ and we deduce from 
\eqref{e:relopsDR}
that
\begin{align}
P_{\rvv}\circ 
J_{\WW}(\hat{x},\hat{u})&=\Lambda^*(\id+\Upsilon^2)^{-1}
\Lambda(J_{\WW}(\hat{x},\hat{u}))
\nonumber\\
&= 
\Lambda^*(\id+\Upsilon^2)^{-1}G_{\Upsilon,B,A}
(\Lambda(\hat{x},\hat{u}))\nonumber\\ 
&= \Lambda^*(\id+\Upsilon^2)^{-1}G_{\Upsilon,B,A}\hat{z}\nonumber\\ 
&=\Lambda^*(\id+\Upsilon^2)^{-1} \hat{z}\nonumber\\
&=(\hat{x},\hat{u}).
\end{align} 
Consequently, $(\hat{x},\hat{u})\in\fix(P_{\rvv}\circ 
J_{\WW})$ and 
 $\hat{z}=\Lambda(\hat{x},\hat{u})\in
 \Lambda(\fix(P_{\rvv}\circ 
J_{\WW}))$. 
  \end{proof}

\end{document}